\documentclass[12pt,reqno]{article}

\usepackage[usenames]{color}
\usepackage[colorlinks=true,
linkcolor=webgreen, filecolor=webbrown,
citecolor=webgreen]{hyperref}

\definecolor{webgreen}{rgb}{0,.5,0}
\definecolor{webbrown}{rgb}{.6,0,0}

\usepackage{amssymb}
\usepackage{graphicx}
\usepackage{amscd}
\usepackage{lscape}
\usepackage{tikz}
\usepackage{tikz-cd}
\usepackage{pgfplots}

\usetikzlibrary{matrix}
\usetikzlibrary{fit,shapes}
\usetikzlibrary{positioning, calc}
\tikzset{circle node/.style = {circle,inner sep=1pt,draw, fill=white},
        X node/.style = {fill=white, inner sep=1pt},
        dot node/.style = {circle, draw, inner sep=5pt}
        }
\usepackage{tkz-fct}

\usepackage{amsthm}
\newtheorem{theorem}{Theorem}
\newtheorem{lemma}[theorem]{Lemma}
\newtheorem{proposition}[theorem]{Proposition}
\newtheorem{corollary}[theorem]{Corollary}
\newtheorem{conjecture}[theorem]{Conjecture}

\theoremstyle{definition}

\newtheorem{example}[theorem]{Example}

\usepackage{float}

\usepackage{graphics,amsmath}
\usepackage{amsfonts}
\usepackage{latexsym}
\usepackage{epsf}

\newcommand{\seqnum}[1]{\href{http://oeis.org/#1}{\underline{#1}}}

\begin{document}

\begin{center}
\vskip 1cm{\LARGE\bf Notes on the Hankel transform of linear combinations of consecutive pairs of Catalan numbers} \vskip 1cm \large
Paul Barry\\
School of Science\\
Waterford Institute of Technology\\
Ireland\\
\href{mailto:pbarry@wit.ie}{\tt pbarry@wit.ie}
\end{center}
\vskip .2 in

\begin{abstract} We provide a context around a conjectured closed form for the Hankel transform of linear combinations of consecutive pairs of Catalan numbers. This generalizes the formula for the Hankel transforms of the shifted Catalan numbers and the known results for linear combinations of pairs of Catalan numbers. Many interesting number triangles emerge, some of which we analyze using the language of Riordan arrays.
\end{abstract}

\section{Introduction}
The Catalan numbers $C_n=\frac{1}{n}\binom{n+1}{n}$, which begin 
 $$1,1,2,5,14,42,132,\ldots,$$ are among the most well known of integer sequences, as evidenced \cite{Stanley} by their many combinatorial applications. The sequence of the Catalan numbers is a moment sequence, as seen in the integral expression 
  $$C_n=\frac{1}{2\pi} \int_{0}^4 x^n \frac{\sqrt{x(4-x)}}{x}\,dx,$$ and thus the generating function $$c(x)=\frac{1-\sqrt{1-4x}}{2x}$$ of this sequence can be represented by a Jacobi type continued fraction. Specifically, this takes the form
$$c(x)=\cfrac{1}{1-x-\cfrac{x^2}{1-2x-\cfrac{x^2}{1-2x-\cfrac{x^2}{1-\cdots}}}}.$$
A consequence of this is that the Hankel transform $h_n$ of the Catalan numbers, defined here as $h_n=|C_{i,j}|_{0 \le i,j \le n}$,  is given by $h_n=1$ for all $n \in \mathbb{N}$. It happens that many Hankel transforms related to the Catalan numbers are tractable, being able to be analyzed and recognized in familiar forms, something which is emphatically not the case for most sequences using methods available at this stage.

For instance, the Hankel transform of the once-shifted sequence of Catalan numbers $C_{n+1}$ is once again given by $h_n=1$, and indeed this property characterizes the Catalan numbers. Note that $C_{n+1}$ is also a moment sequence, with
 $$C_{n+1}=\frac{1}{2\pi} \int_0^4 x^n \sqrt{x(4-x)}\,dx.$$ 
 The Hankel transform of the twice-shifted Catalan numbers $C_{n+2}$ is the sequence $2,3,4,5, \ldots$, and in general, it is known that the Hankel transform of $C_{n+k}$ is given by \cite{Desainte, Forrester, Kratt}
\begin{equation}\label{E1}h_n^{(k)} = |C_{i+j+k}|_{0 \le i,j \le n}=\prod_{i=1}^{k-1} \prod_{j=1}^{i} \frac{2n+2+i+j}{i+j}.\end{equation}
The following matrix  has as its rows the first $7$ such Hankel transforms.
$$\left(
\begin{array}{ccccccc}
 1 & 1 & 1 & 1 & 1 & 1 & \ldots \\
 1 & 1 & 1 & 1 & 1 & 1 & \ldots \\
 2 & 3 & 4 & 5 & 6 & 7 & \ldots \\
 5 & 14 & 30 & 55 & 91 & 140 & \ldots \\
 14 & 84 & 330 & 1001 & 2548 & 5712 & \ldots \\
 42 & 594 & 4719 & 26026 & 111384 & 395352 & \ldots \\
 132 & 4719 & 81796 & 884884 & 6852768 & 41314284 & \ldots \\
\end{array}
\right).$$
The generating functions $g^{(k)}(x)$ ($k=1, \ldots,7$) of these Hankel transform sequences can be deduced from the following list of generating functions, which record the generating functions $1+xg^{(k)}(x)$.
$$\frac{1}{1-x},$$
$$\frac{1}{1-x},$$
$$\frac{1}{(1-x)^2},$$
$$\frac{1+x}{(1-x)^4},$$
$$\frac{1+7x+7x^2+x^3}{(1-x)^7},$$
$$\frac{1+31 x+187 x^2+330 x^3+187 x^4+31 x^5+x^6}{(1-x)^{11}},$$
$$\frac{1+116 x+2727 x^2+21572 x^3+70328 x^4+103376 x^5+70328 x^6+21572 x^7+2727
   x^8+116 x^9+x^{10}}{(1-x)^{16}}.$$
$$ \cdots.$$
A first generalization of this is to look at the Hankel transforms of sums of consecutive pairs of Catalan numbers
$$h_n^{(k),p} = |C_{i+j+k}+C_{i+j+k+1}|_{0 \le i,j \le n}.$$
The following matrix has as its rows the first $7$ such Hankel transforms.
$$\left(
\begin{array}{ccccccc}
 2 & 5 & 13 & 34 & 89 & 233 & \ldots \\
 3 & 8 & 21 & 55 & 144 & 377 & \ldots \\
 7 & 31 & 115 & 390 & 1254 & 3893 & \ldots \\
 19 & 170 & 1075 & 5580 & 25529 & 107036 & \ldots \\
 56 & 1140 & 13915 & 124579 & 906472 & 5687928 & \ldots \\
 174 & 8745 & 225511 & 3813082 & 48173784 & 491753934 & \ldots \\
 561 & 73931 & 4316598 & 148118620 & 3489574428 & 62113595742 & \ldots \\
\end{array}.
\right).$$
The first row is given by the Fibonacci numbers $F_{2n+3}$ \cite{CRI}, while the second row contains the Fibonacci numbers $F_{2n+4}$.
The respective generating functions $g^{(k),p}$ of these sequences can be deduced from the following list of the generating functions  $1+g^{(k),p}$.
$$\frac{1-x}{1-3x+x^2},$$
$$\frac{1}{1-3x+x^2},$$
$$\frac{1+x}{(1-3x+x^2)^2},$$
$$\frac{1+7x-7x^3-x^4}{(1-3x+x^2)^4},$$
$$\frac{1+35x+160x^2-120x^3-371x^4+371x^5+120x^6-160x^7-35x^8-x^9}{(1-3x+x^2)^7},$$
$$\cdots.$$
For instance, the Hankel transform of $C_{n+2}+C_{n+3}$ has its generating function given by 
$$[x^{n+1}] \frac{1+x}{(1-3x+x^2)^2}.$$ 
A natural generalization is to consider the Hankel transforms of the linear combinations
$$ a C_{n+m} + b C_{n+m+1}.$$
Initial results are obtained in \cite{Dougherty, Raj}. For instance, the Hankel transform of $a C_n+b C_{n+1}$ is given by
\begin{equation}\label{D1}[x^{n+1}] \frac{1-bx}{1-(a+2b)x+b^2x^2}.\end{equation}
This is the case of $m=0$. For $m=1,2,3$ we find that the Hankel transforms are given by, respectively,
$$[x^{n+1}] \frac{1}{1-(a+2b)x+b^2x^2},$$
$$[x^{n+1}] \frac{1+bx}{(1-(a+2b)x+b^2x^2)^2},$$
and
$$[x^{n+1}] \frac{(1-b^2x^2)(1+(a+6b)x+b^2x^2)}{(1-(a+2b)x+b^2x^2)^4}.$$
The powers appearing in the denominators are conjectured to be the central polygonal numbers $\frac{m(m-1)}{2}+1$. As an example, the Hankel transform of $2 C_{n+5}+ 3 C_{n+6}$ is given by
$$[x^{n+1}] \frac{1+392 x +26818x^2+\cdots+1874923848x^{15}+43046721x^{16}}{(1-8x+9x^2)^{11}}.$$
\begin{example} When $a=b=1$, the Hankel transform $h_n$ of $a C_{n+1}+ bC_{n+2}$ is given by
$$[x^{n+1}] \frac{1}{1-3x+x^2},$$ while that of $a C_{n+2}+ b C_{n+3}$ is given by
$$[x^{n+1}] \frac{1+x}{(1-3x+x^2)^2}.$$
We now use the fact that the product of two generating functions expands to the convolution of the corresponding sequences to arrive at the following formula for $H_n$, the Hankel transform of $C_{n+2}+ C_{n+3}$.
$$H_n=\sum_{k=0}^{n+1} \binom{1}{k}\sum_{i=0}^{n-k+1} \tilde{h}_i \tilde{h}_{n-k-i+1},$$ where
$\tilde{h}_n$ is the sequence $1,h_0,h_1,h_2,\ldots$.

The general formula for arbitrary $a$ and $b$ is
$$H_n=\sum_{k=0}^{n+1}(b \binom{1}{k}-(b-1)\binom{0}{k})\sum_{i=0}^{n-k+1} \tilde{h}_i \tilde{h}_{n-k-i+1}.$$
\end{example}
The central point of interest of this note is the following conjecture, which generalizes Eq. (\ref{E1}).
\begin{conjecture} The Hankel transform of $a C_{n+m}+ b C_{n+m+1}$ is given by
$$\sum_{k=0}^{n+1} T_{n,k,m} b^{n-k+1} a^k,$$
where
$$T_{n,k,m}=\frac{C_m \binom{m+k-2}{m-2}\binom{n+k+2m-2}{2k+2m-3}\prod_{j=0}^{\lfloor \frac{m-1}{2} \rfloor-1}\binom{2n+2m-2j-1}{2m-4j-5}\prod_{j=0}^{m-3}(2m-j-2)}{\binom{2m-2}{2m-3}\prod_{j=0}^{\lfloor \frac{m-1}{2} \rfloor-1} \binom{2m-2j-1}{2m-4j-5} \prod_{j=0}^{m-3}(n+k+2m-j-2)}.$$
\end{conjecture}
All the arrays $(T_{n,k,m})_{0\le n,k \le \infty}$ are Hessenberg in form. An interesting property of the array $(T_{n,k,m})_{0\le n,k \le \infty}$ is that its initial column is equal to the super-diagonal column of the array $(T_{n,k,m+1})_{0\le n,k \le \infty}$. Thus we have
$$T_{n,0,m}=T_{n,n+1,m+1}.$$
These are precisely the Hankel transforms of $C_{n+m}$.
For $m=2,3,4,5$, the arrays $(T_{n,k,m})_{0 \le n,k \le \infty}$ have the following $5 \times 5$ truncations.
$$\left(
\begin{array}{ccccc}
 2 & 1 & 0 & 0 & 0 \\
 3 & 4 & 1 & 0 & 0 \\
 4 & 10 & 6 & 1 & 0 \\
 5 & 20 & 21 & 8 & 1 \\
 6 & 35 & 56 & 36 & 10 \\
\end{array}
\right),\left(
\begin{array}{ccccc}
 5 & 2 & 0 & 0 & 0 \\
 14 & 14 & 3 & 0 & 0 \\
 30 & 54 & 27 & 4 & 0 \\
 55 & 154 & 132 & 44 & 5 \\
 91 & 364 & 468 & 260 & 65 \\
\end{array}
\right),$$
$$\left(
\begin{array}{ccccc}
 14 & 5 & 0 & 0 & 0 \\
 84 & 72 & 14 & 0 & 0 \\
 330 & 495 & 220 & 30 & 0 \\
 1001 & 2288 & 1716 & 520 & 55 \\
 2548 & 8190 & 9100 & 4550 & 1050 \\
\end{array}
\right),\left(
\begin{array}{ccccc}
 42 & 14 & 0 & 0 & 0 \\
 594 & 462 & 84 & 0 & 0 \\
 4719 & 6292 & 2574 & 330 & 0 \\
 26026 & 52052 & 35490 & 10010 & 1001 \\
 111384 & 309400 & 309400 & 142800 & 30940 \\
\end{array}
\right).$$

The first matrix in this list, 
$$\left(
\begin{array}{ccccc}
 2 & 1 & 0 & 0 & 0 \\
 3 & 4 & 1 & 0 & 0 \\
 4 & 10 & 6 & 1 & 0 \\
 5 & 20 & 21 & 8 & 1 \\
 6 & 35 & 56 & 36 & 10 \\
\end{array}
\right)$$ arises by removing the first row from the matrix $\left(\binom{n+k+1}{2k+1}\right)$, which begins 
$$\left(
\begin{array}{ccccccc}
 1 & 0 & 0 & 0 & 0 & 0 & 0 \\
 2 & 1 & 0 & 0 & 0 & 0 & 0 \\
 3 & 4 & 1 & 0 & 0 & 0 & 0 \\
 4 & 10 & 6 & 1 & 0 & 0 & 0 \\
 5 & 20 & 21 & 8 & 1 & 0 & 0 \\
 6 & 35 & 56 & 36 & 10 & 1 & 0 \\
 7 & 56 & 126 & 120 & 55 & 12 & 1 \\
\end{array}
\right).$$ 
This matrix represents the element $\left(\frac{1}{(1-x)^2}, \frac{x}{(1-x)^2}\right)$ of the Riordan group. 
The second triangle,
$$\left(
\begin{array}{ccccc}
 5 & 2 & 0 & 0 & 0 \\
 14 & 14 & 3 & 0 & 0 \\
 30 & 54 & 27 & 4 & 0 \\
 55 & 154 & 132 & 44 & 5 \\
 91 & 364 & 468 & 260 & 65 \\
\end{array}
\right)$$ arises in the same way from the matrix $\left(\frac{(2n+3)(k+1)}{2k+3}\binom{n+k+2}{2k+2}\right)$. Dividing the columns of this matrix by $k+1$ gives us the matrix $\left(\frac{2n+3}{2k+3}\binom{n+k+2}{2k+2}\right)$, which begins
$$\left(
\begin{array}{ccccccc}
 1 & 0 & 0 & 0 & 0 & 0 & 0 \\
 5 & 1 & 0 & 0 & 0 & 0 & 0 \\
 14 & 7 & 1 & 0 & 0 & 0 & 0 \\
 30 & 27 & 9 & 1 & 0 & 0 & 0 \\
 55 & 77 & 44 & 11 & 1 & 0 & 0 \\
 91 & 182 & 156 & 65 & 13 & 1 & 0 \\
 140 & 378 & 450 & 275 & 90 & 15 & 1 \\
\end{array}
\right).$$ 
This matrix represents the element $\left(\frac{1+x}{(1-x)^4}, \frac{x}{(1-x)^2}\right)$ of the Riordan group.
Returning to the matrix $\left(T_{n,k,4}\right)$ that begins
$$\left(
\begin{array}{ccccc}
 14 & 5 & 0 & 0 & 0 \\
 84 & 72 & 14 & 0 & 0 \\
 330 & 495 & 220 & 30 & 0 \\
 1001 & 2288 & 1716 & 520 & 55 \\
 2548 & 8190 & 9100 & 4550 & 1050 \\
\end{array}
\right),$$ 
we can describe it in terms of scaled versions of the two Riordan arrays $\left(\frac{1+7x+7x^2+x^2}{(1-x)^7}, \frac{x}{(1-x)^2}\right)$ and $\left(\frac{1+x}{(1-x)^7},\frac{x}{(1-x)^2}\right)$ as follows. We have
\begin{scriptsize} 
$$T_{n,k,4}=\left(\binom{k+3}{3}+\binom{k+2}{3}\right)[x^{n+1}] \frac{1+7x+7x^2+x^2}{(1-x)^7}\left(\frac{x}{(1-x)^7}\right)^k-8 \binom{k+2}{3}[x^n]\frac{1+x}{(1-x)^7}\left(\frac{x}{(1-x)^2}\right)^k.$$
\end{scriptsize}
We close this section by noting that the matrix which begins 
$$\left(
\begin{array}{ccccccc}
 1 & 0 & 0 & 0 & 0 & 0 & 0 \\
 14 & 5 & 0 & 0 & 0 & 0 & 0 \\
 84 & 72 & 14 & 0 & 0 & 0 & 0 \\
 330 & 495 & 220 & 30 & 0 & 0 & 0 \\
 1001 & 2288 & 1716 & 520 & 55 & 0 & 0 \\
 2548 & 8190 & 9100 & 4550 & 1050 & 91 & 0 \\
 5712 & 24480 & 37400 & 27200 & 10200 & 1904 & 140 \\
\end{array}
\right)$$ 
has its $(n,k)$-th element given by 
$$[x^{n-k}] \frac{\binom{k+3}{3}+\binom{k+2}{3}+(7\binom{k+3}{3}-\binom{k+2}{3})x+(7\binom{k+3}{3}-\binom{k+2}{3})x^2+(\binom{k+3}{3}+\binom{k+2}{3})x^3}{(1-x)^{2k+7}}.$$
As we have now met several Riordan arrays, we devote the next section to an introduction of the Riordan group for those who may not be familiar with the concept. 
\section{A brief review of Riordan arrays}
We shall use Riordan arrays \cite{book, SGWW} in the sequel as they are well adapted to dealing with the generating functions that are of interest to us. We therefore give a very quick introduction to these special matrices. Many examples of Riordan arrays are to be found in the On-Line Encyclopedia of Integer Sequences (OEIS) \cite{SL1, SL2}. For our purposes, it is sufficient to work with $\mathbb{Q}[[x]]$, the ring of formal power series with rational coefficients. A Riordan array will be the matrix representation of an element $(g, f)$ of $\mathbb{Q}[[x]] \times \mathbb{Q}[[x]]$, where
$$g(x)=g_0+g_1x+g_2 x^2+\cdots,\quad g_0 \ne 0,$$
$$f(x)=f_1 x+ f_2x^2+f_3x^3+\cdots, \quad f_0=0, f_1 \ne 0,$$ where $g_i, f_i \in \mathbb{Q}$. The condition on $g(x)$ ensures that it has a multiplicative inverse, while the condition on $f(x)$ ensures that it is composable with a compositional inverse. The matrix representation of $(g,f)$ is then the matrix $(t_{n,k})_{0 \le i,j \le \infty}$ where
$$t_{n,k}=[x^n] g(x)f(x)^k.$$
Here, $[x^n]$ is the functional that extracts the coefficient of $x^n$ in a formal power series. Note that $x$ plays the role of a ``dummy variable'' in the sense that
$$t_{n,k}=[x^n] g(x)f(x)^k = [t^n] g(t) f(t)^k.$$
There is a group structure on the pairs $(g, f)$ where multiplication is given by
$$(g, f)\cdot (u, v)= (g.u(f), v(f))$$ and inverses are given by
$$(g, f)^{-1}= \left(\frac{1}{g(\bar{f})}, \bar{f}\right)$$ where $\bar{f}$ is the reversion (or compositional inverse) of $f$, that is, where have $f(\bar{f})(x)=x$ and $\bar{f}(f)(x)=x$. In other words, $\bar{f}$ is the solution $u(x)$ of the equation $f(u)=x$ which satisfies $u(0)=0$. The product of pairs $(g, f)$ translates in the matrix representation into ordinary matrix multiplication, and the inverse of a pair $(g, f)$ corresponds to the normal matrix inverse. The identity element is $(1, x)$ whose representation is the identity matrix. This allows us to interpret algebraic operations on power series as linear algebra operations.
The \emph{fundamental theorem of Riordan arrays (FTRA)} is the statement that there is an operation of this group on the ring of power series of the form
$$(g, f)h=g.h(f).$$ Using dummy variables, this appears as $(g(x), f(x))h(x)=g(x)h(f(x))$. In the matrix representation, this is mediated by multiplying the vector whose elements are the coefficients of $h$ by the matrix representing the pair $(g, f)$. Then the elements of the resulting vector are the coefficients of the power series $g.h(f)$.
Each Riordan array is of infinite extent. When we display examples, we shall use a suitable truncation.
\begin{example}
The most famous example of a non-trivial Riordan array is Pascal's triangle $(\binom{n}{k})$. This is the matrix representation of the pair $\left(\frac{1}{1-x}, \frac{x}{1-x}\right)$. Solving
$$\frac{u}{1-u}=x$$ gives $u(x)=\frac{x}{1+x}$ from which we deduce that
$$\left(\frac{1}{1-x}, \frac{x}{1-x}\right)^{-1}=\left(\frac{1}{1+x}, \frac{x}{1+x}\right).$$
This corresponds to the fact that the inverse of $(\binom{n}{k})$ is $((-1)^{n-k}\binom{n}{k})$.
\end{example}
\begin{example} For this paper, the following two Riordan arrays will play an important role. The first is
$$\left(\frac{1}{1+x}, \frac{x}{(1+x)^2}\right).$$ Its matrix representation is given by
$$t_{n,k}=(-1)^{n-k} \binom{n+k}{2k}.$$
The equation
$$\frac{u}{(1+u)^2}=x$$ has for its solution with $u(0)=0$ the following expression.
$$u(x)=\frac{1-2x-\sqrt{1-4x}}{2x}.$$
Then we have $u(x)=\overline{\frac{x}{(1+x)^2}}=c(x)-1$, where $c(x)$ is the generating function of the Catalan numbers.
We find that
$$\left(\frac{1}{1+x}, \frac{x}{(1+x)^2}\right)^{-1}=(c(x), c(x)-1)=(c(x), x c(x)^2).$$
In matrix terms, we have
$$\left(
\begin{array}{ccccccc}
 1 & 0 & 0 & 0 & 0 & 0 & 0 \\
 -1 & 1 & 0 & 0 & 0 & 0 & 0 \\
 1 & -3 & 1 & 0 & 0 & 0 & 0 \\
 -1 & 6 & -5 & 1 & 0 & 0 & 0 \\
 1 & -10 & 15 & -7 & 1 & 0 & 0 \\
 -1 & 15 & -35 & 28 & -9 & 1 & 0 \\
 1 & -21 & 70 & -84 & 45 & -11 & 1 \\
\end{array}
\right)^{-1}=\left(
\begin{array}{ccccccc}
 1 & 0 & 0 & 0 & 0 & 0 & 0 \\
 1 & 1 & 0 & 0 & 0 & 0 & 0 \\
 2 & 3 & 1 & 0 & 0 & 0 & 0 \\
 5 & 9 & 5 & 1 & 0 & 0 & 0 \\
 14 & 28 & 20 & 7 & 1 & 0 & 0 \\
 42 & 90 & 75 & 35 & 9 & 1 & 0 \\
 132 & 297 & 275 & 154 & 54 & 11 & 1 \\
\end{array}
\right).$$

The first column of the inverse is seen to be $C_n$. 
Thus the matrix $\left((-1)^{n-k} \binom{n+k}{2k}\right)$ is directly related to the Catalan numbers. We note in particular that
$$c\left(\frac{x}{(1+x)^2}\right)=1+x,$$ and so we have
$$\left(\frac{1}{1+x}, \frac{x}{(1+x)^2}\right)\cdot c(x)=\frac{1}{1+x} c\left(\frac{x}{(1+x)^2}\right)=\frac{1}{1+x}.(1+x)=1.$$
Note that $1$ as a formal power series is the generating function of the sequence $(1,0,0,0,\ldots)$; this sequence is often denoted by $0^n$. In matrix terms, we have
$$\left(
\begin{array}{ccccccc}
 1 & 0 & 0 & 0 & 0 & 0 & 0 \\
 -1 & 1 & 0 & 0 & 0 & 0 & 0 \\
 1 & -3 & 1 & 0 & 0 & 0 & 0 \\
 -1 & 6 & -5 & 1 & 0 & 0 & 0 \\
 1 & -10 & 15 & -7 & 1 & 0 & 0 \\
 -1 & 15 & -35 & 28 & -9 & 1 & 0 \\
 1 & -21 & 70 & -84 & 45 & -11 & 1 \\
\end{array}
\right)\left(\begin{array}{c}
1\\
1\\
2\\
5\\
14\\
42\\
132\\ \end{array}\right)=
\left(\begin{array}{c}
1\\
0\\
0\\
0\\
0\\
0\\
0\\ \end{array}\right).$$

The second Riordan array of importance is related to the first. It is 
$$\left(\frac{1}{(1+x)^2}, \frac{x}{(1+x)^2}\right),$$ which has the matrix representation 
$$t_{n,k}=(-1)^{n-k} \binom{n+k+1}{2k+1}.$$ This matrix begins 
$$\left(
\begin{array}{ccccccc}
 1 & 0 & 0 & 0 & 0 & 0 & 0 \\
 -2 & 1 & 0 & 0 & 0 & 0 & 0 \\
 3 & -4 & 1 & 0 & 0 & 0 & 0 \\
 -4 & 10 & -6 & 1 & 0 & 0 & 0 \\
 5 & -20 & 21 & -8 & 1 & 0 & 0 \\
 -6 & 35 & -56 & 36 & -10 & 1 & 0 \\
 7 & -56 & 126 & -120 & 55 & -12 & 1 \\
\end{array}
\right),$$ with inverse that begins
$$\left(
\begin{array}{ccccccc}
 1 & 0 & 0 & 0 & 0 & 0 & 0 \\
 2 & 1 & 0 & 0 & 0 & 0 & 0 \\
 5 & 4 & 1 & 0 & 0 & 0 & 0 \\
 14 & 14 & 6 & 1 & 0 & 0 & 0 \\
 42 & 48 & 27 & 8 & 1 & 0 & 0 \\
 132 & 165 & 110 & 44 & 10 & 1 & 0 \\
 429 & 572 & 429 & 208 & 65 & 12 & 1 \\
\end{array}
\right).$$ We see that the first column of the inverse is given by $C_{n+1}$. 
\end{example}

We shall write $M=\left((-1)^{n-k} \binom{n+k}{2k}\right)$ and $\tilde{M}=\left((-1)^{n-k} \binom{n+k+1}{2k+1}\right)$. We shall also use the notation 
$$\mathcal{M}=\left(\frac{1}{1+x}, \frac{x}{(1+x)^2}\right)$$ 
and 
$$\tilde{\mathcal{M}}=\left(\frac{1}{(1+x)^2}, \frac{x}{(1+x)^2}\right)$$ as elements of the Riordan group.

\section{The case of $a C_n+ b C_{n+1}$}
In order to find a closed form for the Hankel transform of $a C_n+ b C_{n+1}$, we proceed as follows. 
\begin{lemma} We have
$$\left(\frac{1}{1-bx}, \frac{x}{(1-bx)^2}\right)\cdot \frac{1}{1-ax}=\frac{1-bx}{1-(a+2b)x+b^2x^2}.$$
\end{lemma}
\begin{proof} By the fundamental theorem of Riordan arrays, we have
\begin{align*}\left(\frac{1}{1-bx}, \frac{x}{(1-bx)^2}\right)\cdot \frac{1}{1-ax}&=
\frac{1}{1-bx} \frac{1}{1-\frac{ax}{(1-bx)^2}}\\
&=\frac{1}{1-bx} \frac{(1-bx)^2}{(1-bx)^2-ax}\\
&=\frac{1-bx}{1-(a+2b)x+b^2x^2}.\end{align*}
\end{proof}
The Riordan array $\left(\frac{1}{1-bx}, \frac{x}{(1-bx)^2}\right)$ has general term $\binom{n+k}{2k}b^{n-k}$.
\begin{proposition} The Hankel transform of $aC_n+bC_{n+1}$ is given by
$$h_n=\sum_{k=0}^{n+1} \binom{n+k+1}{2k}b^{n-k+1}a^k.$$
\end{proposition}
\begin{proof}
By the lemma, this follows from Eq (\ref{D1}).
\end{proof}
The generating function $A(x)$ of $aC_n+bC_{n+1}$ is given by $$A(x)=\frac{1}{x}(c(x)(ax+b)-b).$$
Then $$ \frac{xA(x)-yA(y)}{x-y}$$ is the bivariate generating function of the Hankel matrix $\mathcal{H}$ of  $aC_n+bC_{n+1}$. Now the matrix $((-1)^{n-k} \binom{n+k}{2k})$ is lower triangular with $1$'s on the diagonal, so that the sequence of principal minors of the matrix whose bivariate generating function is given by 
$$\left(\frac{1}{1+x},\frac{x}{(1+x)^2}\right)\cdot \mathcal{H}\cdot \left(\frac{1}{1+y},\frac{y}{(1+y)^2}\right)^T$$ will be equal to the Hankel transform of $aC_n+bC_{n+1}$.
Thus we calculate 
$$\frac{1}{(1+x)(1+y)}\frac{\frac{x}{(1+x)^2}A\left(\frac{x}{(1+x)^2}\right)-\frac{y}{(1+y)^2}A\left(\frac{y}{(1+y)^2}\right)}{\frac{x}{(1+x)^2}-\frac{y}{(1+y)^2}}$$ to obtain the generating function 
$$\frac{a+b+bx(1+y)+by}{1-xy}.$$ 
This expands to give the tri-diagonal matrix which begins 
$$\left(
\begin{array}{ccccccc}
 a+b & b & 0 & 0 & 0 & 0 & 0 \\
 b & a+2 b & b & 0 & 0 & 0 & 0 \\
 0 & b & a+2 b & b & 0 & 0 & 0 \\
 0 & 0 & b & a+2 b & b & 0 & 0 \\
 0 & 0 & 0 & b & a+2 b & b & 0 \\
 0 & 0 & 0 & 0 & b & a+2 b & b \\
 0 & 0 & 0 & 0 & 0 & b & a+2 b \\
\end{array}
\right).$$ 
The principal minor sequence of this matrix is then given by 
$$h_n=[x^{n+1}]\frac{1-bx}{1-(a+2b)x+b^2 x^2}.$$

\section{The case of $a C_{n+1}+ b C_{n+2}$}
\begin{lemma} We have
$$\left(\frac{1}{(1-bx)^2}, \frac{x}{(1-bx)^2}\right)\cdot \frac{1}{1-ax}=\frac{1}{1-(a+2b)x+b^2x^2}.$$
\end{lemma}
\begin{proof} As in the previous lemma, this follows from the FTRA.
\end{proof}
The Riordan array $\left(\frac{1}{(1-bx)^2}, \frac{x}{(1-bx)^2}\right)$ has general term $(\binom{n+k+1}{2k})b^{n-k}$. We thus obtain the following result.
\begin{proposition}
The Hankel transform of $a C_{n+1}+ bC_{n+2}$ is given by
$$h_n=\sum_{k=0}^{n+1} \binom{n+k+2}{2k+1}b^{n-k+1}a^k.$$
\end{proposition}
We now note that we have the factorization of Riordan arrays
$$\left(\frac{1}{(1-bx)^2}, \frac{x}{(1-bx)^2}\right)=\left(\frac{1}{1-bx}, x\right)\cdot \left(\frac{1}{1-bx}, \frac{x}{(1-bx)^2}\right).$$
We deduce the following.
\begin{proposition}
Let $h_n$ be the Hankel transform of $a C_n+ b C_{n+1}$. Let $H_n$ be the Hankel transform of $a C_{n+1}+ b C_{n+2}$. Let $\tilde{h}_n$ be the sequence $1, h_0, h_1,\ldots$ and let $\tilde{H}_n$ be the sequence $1, H_0, H_1,\ldots$. Then we have
$$\tilde{H}_n=\sum_{k=0}^n b^{n-k}\tilde{H}_k$$ and hence
$$H_n = \sum_{k=0}^{n+1} b^{n-k+1}\tilde{H}_k.$$
\end{proposition}
\begin{proof}
In effect, the $(n,k)$-th element of the Riordan array $\left(\frac{1}{1-bx},x\right)$ is $b^{n-k}$.
\end{proof}
The generating function $A(x)$ of $aC_{n+1}+bC_{n+2}$ is given by $$A(x)=\frac{1}{x^2}\left(ct(x)(ax+b)-(ax+bx+b)\right).$$
Then $$ \frac{xA(x)-yA(y)}{x-y}$$ is the bivariate generating function of the Hankel matrix $\mathcal{H}$ of  $aC_n+bC_{n+1}$. Now the matrix $\tilde{M}=((-1)^{n-k} \binom{n+k+1}{2k+1})$ is lower triangular with $1$'s on the diagonal, so that the sequence of principal minors of the matrix whose bivariate generating function is given by
$$\left(\frac{1}{(1+x)^2},\frac{x}{(1+x)^2}\right)\cdot \mathcal{H}\cdot \left(\frac{1}{(1+y)^2},\frac{y}{(1+y)^2}\right)^T$$ will be equal to the Hankel transform of $aC_{n+1}+bC_{n+2}$.
Thus we calculate
$$\frac{1}{(1+x)^2(1+y)^2}\frac{\frac{x}{(1+x)^2}A\left(\frac{x}{(1+x)^2}\right)-\frac{y}{(1+y)^2}A\left(\frac{y}{(1+y)^2}\right)}{\frac{x}{(1+x)^2}-\frac{y}{(1+y)^2}}$$ to obtain the generating function
$$\frac{a+2b+bx+by}{1-xy}.$$
This expands to give the tri-diagonal matrix which begins
$$\left(
\begin{array}{ccccccc}
 a+2b & b & 0 & 0 & 0 & 0 & 0 \\
 b & a+2 b & b & 0 & 0 & 0 & 0 \\
 0 & b & a+2 b & b & 0 & 0 & 0 \\
 0 & 0 & b & a+2 b & b & 0 & 0 \\
 0 & 0 & 0 & b & a+2 b & b & 0 \\
 0 & 0 & 0 & 0 & b & a+2 b & b \\
 0 & 0 & 0 & 0 & 0 & b & a+2 b \\
\end{array}
\right).$$
The principal minor sequence of this matrix is then given by
$$h_n=[x^{n+1}]\frac{1}{1-(a+2b)x+b^2 x^2}.$$ 

We now make the observation, using the notation $H(n+r, a,b)$ for the Hankel matrix of $aC_{n+r}+bC_{n+r+1}$, that the matrix
$$ M\cdot H(n,a,b)\cdot M^T-\tilde{M}\cdot H(n+1,a,b)\cdot \tilde{M}^T$$ begins
$$\left(
\begin{array}{ccccccc}
 b & 0 & 0 & 0 & 0 & 0 & 0 \\
 0 &  0 & 0 & 0 & 0 & 0 & 0 \\
 0 & 0 & 0 & 0 & 0 & 0 & 0 \\
 0 & 0 & 0 & 0 & 0 & 0 & 0 \\
 0 & 0 & 0 & 0 &  0 & 0 & 0 \\
 0 & 0 & 0 & 0 & 0 & 0 & 0 \\
 0 & 0 & 0 & 0 & 0 & 0 & 0 \\
\end{array}
\right).$$ 
This is the Hankel matrix of the sequence $b,0,0,0,\ldots$. 

\section{The case of $a C_{n+2}+b C_{n+3}$}
By another use of the FTRA, we obtain the following result.
\begin{lemma} We have

$$\left(\frac{1+bx}{(1-bx)^4}, \frac{x}{(1-bx)^2}\right)\cdot \frac{1}{(1-ax)^2}=\frac{1-bx}{(1-(a+2b)x+b^2x^2)^2}.$$
\end{lemma}
Now the generating function $\frac{1}{(1-ax)^2}$ expands to give the sequence $1,2a,3a^2,4a^3,\ldots$.
We have, for instance,
$$\left(
\begin{array}{cccccc}
 5 b & 1 & 0 & 0 & 0 & 0 \\
 14 b^2 & 7 b & 1 & 0 & 0 & 0 \\
 30 b^3 & 27 b^2 & 9 b & 1 & 0 & 0 \\
 55 b^4 & 77 b^3 & 44 b^2 & 11 b & 1 & 0 \\
 91 b^5 & 182 b^4 & 156 b^3 & 65 b^2 & 13 b & 1 \\
 140 b^6 & 378 b^5 & 450 b^4 & 275 b^3 & 90 b^2 & 15 b \\
\end{array}
\right)\cdot \left(
\begin{array}{cccccc}
 1 & 0 & 0 & 0 & 0 & 0 \\
 0 & 2 & 0 & 0 & 0 & 0 \\
 0 & 0 & 3 & 0 & 0 & 0 \\
 0 & 0 & 0 & 4 & 0 & 0 \\
 0 & 0 & 0 & 0 & 5 & 0 \\
 0 & 0 & 0 & 0 & 0 & 6 \\
\end{array}
\right)\cdot \left(\begin{array}{c}
1\\
a\\
a^2\\
a^3\\
a^4\\
a^5\\
\end{array}\right)$$
$$=\left(\begin{array}{c}
2 a + 5 b \\ 3 a^2 + 14 a b + 14 b^2 \\
4 a^3 + 27 a^2 b + 54 a b^2 + 30 b^3 \\
5 a^4 + 44 a^3 b + 132 a^2 b^2 + 154 a b^3 + 55 b^4 \\
6 a^5 + 65 a^4 b + 260 a^3 b^2 + 468 a^2 b^3 + 364 a b^4 + 91 b^5 \\
7 a^6 + 90 a^5 b + 450 a^4 b^2 + 1100 a^3 b^3 + 1350 a^2 b^4 +
 756 a b^5 + 140 b^6 \\
 \end{array}\right).
 $$
Thus the matrix $(T_{n,k,3})$ is obtained in the following way.
$$\left(
\begin{array}{cccccc}
 5  & 1 & 0 & 0 & 0 & 0 \\
 14  & 7  & 1 & 0 & 0 & 0 \\
 30  & 27  & 9  & 1 & 0 & 0 \\
 55  & 77  & 44  & 11 & 1 & 0 \\
 91  & 182  & 156  & 65  & 13 & 1 \\
 140  & 378  & 450  & 275  & 90  & 15  \\
\end{array}
\right)\cdot \left(
\begin{array}{cccccc}
 1 & 0 & 0 & 0 & 0 & 0 \\
 0 & 2 & 0 & 0 & 0 & 0 \\
 0 & 0 & 3 & 0 & 0 & 0 \\
 0 & 0 & 0 & 4 & 0 & 0 \\
 0 & 0 & 0 & 0 & 5 & 0 \\
 0 & 0 & 0 & 0 & 0 & 6 \\
\end{array}
\right)=\left(
\begin{array}{cccccc}
 5 & 2 & 0 & 0 & 0 & 0 \\
 14 & 14 & 3 & 0 & 0 & 0 \\
 30 & 54 & 27 & 4 & 0 & 0 \\
 55 & 154 & 132 & 44 & 5 & 0 \\
 91 & 364 & 468 & 260 & 65 & 6 \\
 140 & 756 & 1350 & 1100 & 450 & 90 \\
\end{array}
\right).$$
\begin{proposition} The Hankel transform of $aC_{n+2}+ bC_{n+3}$ is given by
$$h_n = \sum_{k=0}^{n+1}\frac{2n+5}{n+k+4} \binom{n+k+4}{n-k+1}b^{n-k+1}(k+1)a^k.$$
\end{proposition}
\begin{proof} This follows since the general $(n,k)$-term of the Riordan array
$$\left(\frac{1+bx}{(1-bx)^4}, \frac{x}{(1-bx)^2}\right)$$ is given by
$$\frac{2n+3}{n+k+3} \binom{n+k+3}{n-k}.$$
\end{proof}

A missing element of our work is the proof of the generating function of the Hankel transform of $aC_{n+2}+b C_{n+3}$. We now provide this. First, we need the following lemma.

\begin{lemma} The principal minor sequence of the pentadiagonal matrix  that begins
$$\left(
\begin{array}{ccccccc}
 a-r & b & c & 0 & 0 & 0 & 0 \\
 b & a & b & c & 0 & 0 & 0 \\
 c & b & a & b & c & 0 & 0 \\
 0 & c & b & a & b & c & 0 \\
 0 & 0 & c & b & a & b & c \\
 0 & 0 & 0 & c & b & a & b \\
 0 & 0 & 0 & 0 & c & b & a \\
\end{array}
\right)$$ has generating function
$$\frac{1-(r-c)x-rx^2}{1-(a-c)x-c(ac-b^2)x^2+c(ac-b^2)x^3+(a-c)x^4-x^5}.$$
\end{lemma}
We begin by looking at the case of $C_{n+2}+C_{n+3}$.
\begin{proposition} The Hankel transform of $C_{n+2}+ C_{n+3}$ is given by the principal minor sequence of the pentadiagonal matrix that begins
$$\left(
\begin{array}{ccccccc}
 7 & 5 & 1 & 0 & 0 & 0 & 0 \\
 5 & 8 & 5 & 1 & 0 & 0 & 0 \\
 1 & 5 & 8 & 5 & 1 & 0 & 0 \\
 0 & 1 & 5 & 8 & 5 & 1 & 0 \\
 0 & 0 & 1 & 5 & 8 & 5 & 1 \\
 0 & 0 & 0 & 1 & 5 & 8 & 5 \\
 0 & 0 & 0 & 0 & 1 & 5 & 8 \\
\end{array}
\right).$$
\end{proposition}
\begin{proof}
The generating function $A(x)$ of $C_{n+2}+C_{n+3}$ is given by
$$A(x)=\frac{1}{x^3}((1+x)c(x)-(1+2x+3x^2)).$$
Then the bivariate generating function of the Hankel matrix for $C_{n+2}+C_{n+3}$ is given by
$$\frac{xA(x)-yA(y)}{x-y}.$$
The generating function of
$$\left(\frac{1}{(1+x)^2},\frac{x}{(1+x)^2}\right)\cdot \mathcal{H}\cdot \left(\frac{1}{(1+y)^2},\frac{y}{(1+y)^2}\right)^T$$ (whose minor sequence is equal to the Hankel transform sought) is then given by
$$\frac{1}{(1+x)^2(1+y)^2}\frac{\frac{x}{(1+x)^2}A\left(\frac{x}{(1+x)^2}\right)-\frac{y}{(1+y)^2}A\left(\frac{y}{(1+y)^2}\right)}{\frac{x}{(1+x)^2}-\frac{y}{(1+y)^2}}.$$
Simplifying, we find the bivariate generating function
$$\frac{7+5y+y^2+x(y+5)+x^2}{1-xy}$$ of the above pentadiagonal matrix.
\end{proof}
\begin{corollary} The Hankel transform $h_n$ of $C_{n+2}+C_{n+3}$ is given by
$$h_n=[x^{n+1}]\frac{1+x}{1-6x+11x^2-6x^3+x^4}.$$
\end{corollary}
\begin{proof} This is the case $a=5, b=8, c=1, r=1$ of the above lemma. Thus
$$h_n=[x^n] \frac{1-x^2}{1-7x+17x^2-17x^3+7x^4-x^5}=[x^{n+1}]\frac{1+x}{1-6x+11x^2-6x^3+x^4}.$$
\end{proof}
We now turn to the general case.
\begin{proposition} The Hankel transform $h_n$ of $aC_{n+2}+bC_{n+3}$ is given by
$$h_n=[x^{n+1}]\frac{1-b^2x^2}{1-(2a+5b)x+(a^2+6ab+10b^2)x^2-(a^2+6ab+10b^2)bx^3+(2a+5b)c^3x^4-b^5x^5}.$$
\end{proposition}
\begin{proof}
The generating function $A(x)$ of $aC_{n+2}+bC_{n+3}$ is given by
$$A(x)=\frac{1}{x^3}\left((ax+b)c(x)-ax(1+x)-b(1+x+2x^2)\right).$$
We use this to calculate the bivariate generating function
$$\frac{1}{(1+x)^2(1+y)^2}\frac{\frac{x}{(1+x)^2}A\left(\frac{x}{(1+x)^2}\right)-\frac{y}{(1+y)^2}A\left(\frac{y}{(1+y)^2}\right)}{\frac{x}{(1+x)^2}-\frac{y}{(1+y)^2}}.$$
This simplifies to
$$\frac{bx^2+x(by+a+4b)+by^2+y(a+4b)+2a+5b}{1-xy}.$$
This bivariate generating function expands to give the pentadiagonal matrix $\tilde{M} \cdot H(n+2,a,b) \cdot \tilde{M}^T$ that begins
$$\left(
\begin{array}{cccccc}
 2 a+5 b & a+4 b & b & 0 & 0 & 0 \\
 a+4 b & 2a+6 b & a+4 b & b & 0 & 0 \\
 b & a+4 b & 2a+6 b & a+4 b & b & 0 \\
 0 & b & a+4 b & 2a+6 b & a+4 b & b \\
 0 & 0 & b & a+4 b & 2 a+6 b & a+4 b \\
 0 & 0 & 0 & b & a+4 b & 2 a+6 b \\
\end{array}
\right).$$
Thus we have $r \to b$, $a \to 2a+6b$, $b \to a+4$, and $c \to b$. The result follows from this.
\end{proof}
We finish this section by noting that the matrix
$$ \tilde{M}\cdot H(n+2,a,b)\cdot \tilde{M}^T-M\cdot H(n+1,a,b)\cdot M^T $$ begins
$$\left(
\begin{array}{ccccccc}
 a+3b & b & 0 & 0 & 0 & 0 & 0 \\
 b &  0 & 0 & 0 & 0 & 0 & 0 \\
 0 & 0 & 0 & 0 & 0 & 0 & 0 \\
 0 & 0 & 0 & 0 & 0 & 0 & 0 \\
 0 & 0 & 0 & 0 &  0 & 0 & 0 \\
 0 & 0 & 0 & 0 & 0 & 0 & 0 \\
 0 & 0 & 0 & 0 & 0 & 0 & 0 \\
\end{array}
\right).$$ 
This is the Hankel matrix of the sequence $a+3b, b, 0,0,0,\ldots$. 

\section{The case of $a C_{n+3}+b C_{n+4}$}
For the case of $C_{n+3}+C_{n+4}$, we obtain the $7$-diagonal matrix that begins
$$\left(
\begin{array}{ccccccc}
 19 & 18 & 7 & 1 & 0 & 0 & 0 \\
 18 & 26 & 19 & 7 & 1 & 0 & 0 \\
 7 & 19 & 26 & 19 & 7 & 1 & 0 \\
 1 & 7 & 19 & 26 & 19 & 7 & 1 \\
 0 & 1 & 7 & 19 & 26 & 19 & 7 \\
 0 & 0 & 1 & 7 & 19 & 26 & 19 \\
 0 & 0 & 0 & 1 & 7 & 19 & 26 \\
\end{array}
\right),$$ from which we infer that the Hankel transform of $C_{n+3}+C_{n+4}$ is given by
$$[x^n]\frac{(1-x^2)(1+7x+x^2)}{(1-3x+x^2)^4}.$$ For the general case $a C_{n+3}+b C_{n+4}$ we have that
$$A(x)=\frac{1}{x^4}\left((ax+b)c(x)-ax(1+x+2x^2)-b(1+x+2x^2+5x^3)\right).$$
Proceeding as before, we find the following bivariate generating function
$$\frac{b x^3 + x^2(b y + a + 6b) + x(b y^2 + y(a + 6b) + 2(2a + 7b)) + by^3 + y^2(a + 6b) + 2y(2a + 7b) + 5a + 14b}{1-xy},$$
which expands to give the $7$-diagonal matrix that begins
$$\left(
\begin{array}{ccccccc}
 5 a+14 b & 4 a+14 b & a+6 b & b & 0 & 0 & 0 \\
 4 a+14 b & 6 a+20 b & 4 a+15 b & a+6 b & b & 0 & 0 \\
 a+6 b & 4 a+15 b & 6 a+20 b & 4 a+15 b & a+6 b & b & 0 \\
 b & a+6 b & 4 a+15 b & 6 a+20 b & 4 a+15 b & a+6 b & b \\
 0 & b & a+6 b & 4 a+15 b & 6 a+20 b & 4 a+15 b & a+6 b \\
 0 & 0 & b & a+6 b & 4 a+15 b & 6 a+20 b & 4 a+15 b \\
 0 & 0 & 0 & b & a+6 b & 4 a+15 b & 6 a+20 b \\
\end{array}
\right).$$ From this we infer that the Hankel transform of $aC_{n+3}+bC_{n+4}$ is given by
$$[x^n] \frac{(1-b^2x^2)(1+(a+6b)x+b^2x^2)}{(1-(a+2b)x+b^2x^2)^4}.$$
\begin{proposition} The Hankel transform of $a C_{n+3}+ b C_{n+4}$ is given by
$$h_n=[x^{n+1}]\left(\frac{(1-b^2 x^2)(1+(a+6b)x+b^2x^2)}{(1-bx)^8}, \frac{x}{(1-bx)^2}\right)\cdot \frac{1}{(1-ax)^4}.$$
\end{proposition}
\begin{proof} By the FTRA, we have that
$$\left(\frac{(1-b^2 x^2)(1+(a+6b)x+b^2x^2)}{(1-bx)^8}, \frac{x}{(1-bx)^2}\right)\cdot \frac{1}{(1-ax)^4}=
\frac{(1-b^2 x^2)(1+(a+6b)x+b^2x^2)}{(1-(a+2b)x+b^2x^2)^4}.$$
\end{proof}
When $a=b=1$, we obtain
$$h_n=\left(\frac{1+8x+8x^2+x^3}{(1-x)^7}, \frac{x}{(1-x)^2}\right)\cdot \frac{1}{(1-x)^4}.$$
We then have, for instance,
$$\left(
\begin{array}{cccccc}
 1 & 0 & 0 & 0 & 0 & 0 \\
 15 & 1 & 0 & 0 & 0 & 0 \\
 92 & 17 & 1 & 0 & 0 & 0 \\
 365 & 125 & 19 & 1 & 0 & 0 \\
 1113 & 598 & 162 & 21 & 1 & 0 \\
 2842 & 2184 & 903 & 203 & 23 & 1 \\
\end{array}
\right)\left(\begin{array}{c}
1\\ 4\\ 10\\\ 20\\ 35\\ 56\\ \end{array} \right)=
\left(\begin{array}{c}
1\\ 19\\ 170\\\ 1075\\ 5580\\ 25529\\ \end{array} \right).$$
Now the Riordan array $\left(\frac{1+8x+8x^2+x^3}{(1-x)^7}, \frac{x}{(1-x)^2}\right)$ has general $(n,k)$-th term \begin{align}\label{L2}T_{n,k}&=
\binom{n+k+6}{n-k}+8 \binom{n+k+5}{n-k-1}+8\binom{n+k+4}{n-k-2}+\binom{n+k+3}{n-k-3}\\
&=\binom{n+k+6}{n-k}\left(2(n + 2)(5k^2 + 14k - 3(3n^2 + 12n + 10))\right).\end{align}
The generating function $\frac{1}{(1-x)^4}$ expands to give the sequence $1, 4, 10, 20, 35, 56, \ldots$, which is $\binom{n+3}{3}$.
Thus the sequence $1,19,170,1075,\ldots$ is given by
$$\sum_{k=0}^n T_{n,k}\binom{k+3}{3}.$$

Once again using the FTRA, we have the alternative representation
$$h_n=[x^{n+1}]\left(\frac{(1-b^2x^2)(1+(a+6b)x+b^2x^2)}{(1-bx)^6(1+(a-2b)x+b^2x^2)}, \frac{x}{(1-bx)^2}\right)\cdot \frac{1+ax}{(1-ax)^4}.$$
This is of interest since $[x^{n+1}]\frac{1+x}{(1-x)^4}$ is the Hankel transform of $C_{n+2}+C_{n+3}$.
Thus we have, for instance, when $a=1$ and $b=1$,
$$\left(
\begin{array}{cccccc}
 1 & 0 & 0 & 0 & 0 & 0 \\
 14 & 1 & 0 & 0 & 0 & 0 \\
 76 & 16 & 1 & 0 & 0 & 0 \\
 258 & 107 & 18 & 1 & 0 & 0 \\
 657 & 456 & 142 & 20 & 1 & 0 \\
 1380 & 1462 & 722 & 181 & 22 & 1 \\
\end{array}
\right)\cdot \left(\begin{array}{c}
1\\
5\\
14\\
30\\
55\\
91\\
\end{array}\right)=
\left(\begin{array}{c}
1\\
19\\
170\\
1075\\
5580\\
25529\\
\end{array}\right).$$

Here, $5,14,30,55,\ldots$ is the Hankel transform of $C_{n+2}+C_{n+3}$, and $19,170,1075,5580,\ldots$ is the Hankel transform of $C_{n+3}+C_{n+4}$. The Riordan array
$$\left(\frac{(1+x)(1+7x+x^2)}{(1-x)^5(1-x+x^2)}, \frac{x}{(1-x)^2}\right),$$ or
$$\left(\frac{1+8x+8x+x^3}{1-6x+16x^2-25x^3+25x^4-16x^5+6x^6-x^7}, \frac{x}{(1-x)^2}\right)$$ performs the role of transfer matrix from one Hankel transform to the other.

Now we have the factorization
$$\left(\frac{(1+x)(1+7x+x^2)}{(1-x)^5(1-x+x^2)}, \frac{x}{(1-x)^2}\right)=\left(\frac{(1-x)^2}{1-x+x^2},x\right)\cdot \left(\frac{1+8x+8x^2+x^3}{(1-x)^7}, \frac{x}{(1-x)^2}\right).$$
For instance, we have
\begin{scriptsize}
$$\left(
\begin{array}{cccccc}
 1 & 0 & 0 & 0 & 0 & 0 \\
 14 & 1 & 0 & 0 & 0 & 0 \\
 76 & 16 & 1 & 0 & 0 & 0 \\
 258 & 107 & 18 & 1 & 0 & 0 \\
 657 & 456 & 142 & 20 & 1 & 0 \\
 1380 & 1462 & 722 & 181 & 22 & 1 \\
\end{array}
\right)=\left(
\begin{array}{cccccc}
 1 & 0 & 0 & 0 & 0 & 0 \\
 -1 & 1 & 0 & 0 & 0 & 0 \\
 -1 & -1 & 1 & 0 & 0 & 0 \\
 0 & -1 & -1 & 1 & 0 & 0 \\
 1 & 0 & -1 & -1 & 1 & 0 \\
 1 & 1 & 0 & -1 & -1 & 1 \\
\end{array}
\right)\cdot \left(
\begin{array}{cccccc}
 1 & 0 & 0 & 0 & 0 & 0 \\
 15 & 1 & 0 & 0 & 0 & 0 \\
 92 & 17 & 1 & 0 & 0 & 0 \\
 365 & 125 & 19 & 1 & 0 & 0 \\
 1113 & 598 & 162 & 21 & 1 & 0 \\
 2842 & 2184 & 903 & 203 & 23 & 1 \\
\end{array}
\right),$$ \end{scriptsize} or equivalently,
\begin{scriptsize}
$$\left(
\begin{array}{cccccc}
 1 & 0 & 0 & 0 & 0 & 0 \\
 1 & 1 & 0 & 0 & 0 & 0 \\
 2 & 1 & 1 & 0 & 0 & 0 \\
 3 & 2 & 1 & 1 & 0 & 0 \\
 4 & 3 & 2 & 1 & 1 & 0 \\
 5 & 4 & 3 & 2 & 1 & 1 \\
\end{array}
\right)\cdot \left(
\begin{array}{cccccc}
 1 & 0 & 0 & 0 & 0 & 0 \\
 14 & 1 & 0 & 0 & 0 & 0 \\
 76 & 16 & 1 & 0 & 0 & 0 \\
 258 & 107 & 18 & 1 & 0 & 0 \\
 657 & 456 & 142 & 20 & 1 & 0 \\
 1380 & 1462 & 722 & 181 & 22 & 1 \\
\end{array}
\right)=\left(
\begin{array}{cccccc}
 1 & 0 & 0 & 0 & 0 & 0 \\
 15 & 1 & 0 & 0 & 0 & 0 \\
 92 & 17 & 1 & 0 & 0 & 0 \\
 365 & 125 & 19 & 1 & 0 & 0 \\
 1113 & 598 & 162 & 21 & 1 & 0 \\
 2842 & 2184 & 903 & 203 & 23 & 1 \\
\end{array}
\right).$$ \end{scriptsize}
We can also find another relevant factorization. We have
$$ (1+x, x)\cdot \frac{1}{(1-x)^4} = \frac{1+x}{(1-x)^4}.$$ Thus we deduce that
$$\left(\frac{1+8x+8x^2+x^3}{(1-x)^7}, \frac{x}{(1-x)^2}\right)=\left(\frac{(1+x)(1+7x+x^2)}{(1-x)^5(1-x+x^2)}, \frac{x}{(1-x)^2}\right) \cdot (1+x,x).$$
For instance, we have
$$\left(
\begin{array}{cccccc}
 1 & 0 & 0 & 0 & 0 & 0 \\
 15 & 1 & 0 & 0 & 0 & 0 \\
 92 & 17 & 1 & 0 & 0 & 0 \\
 365 & 125 & 19 & 1 & 0 & 0 \\
 1113 & 598 & 162 & 21 & 1 & 0 \\
 2842 & 2184 & 903 & 203 & 23 & 1 \\
\end{array}
\right)=\left(
\begin{array}{cccccc}
 1 & 0 & 0 & 0 & 0 & 0 \\
 14 & 1 & 0 & 0 & 0 & 0 \\
 76 & 16 & 1 & 0 & 0 & 0 \\
 258 & 107 & 18 & 1 & 0 & 0 \\
 657 & 456 & 142 & 20 & 1 & 0 \\
 1380 & 1462 & 722 & 181 & 22 & 1 \\
\end{array}
\right)\cdot \left(
\begin{array}{cccccc}
 1 & 0 & 0 & 0 & 0 & 0 \\
 1 & 1 & 0 & 0 & 0 & 0 \\
 0 & 1 & 1 & 0 & 0 & 0 \\
 0 & 0 & 1 & 1 & 0 & 0 \\
 0 & 0 & 0 & 1 & 1 & 0 \\
 0 & 0 & 0 & 0 & 1 & 1 \\
\end{array}
\right).$$

We finish this section by noting that the matrix 
$$ \tilde{M}\cdot H(n+3,a,b)\cdot \tilde{M}^T-M\cdot H(n+2,a,b)\cdot M^T $$ begins
$$\left(
\begin{array}{ccccccc}
 3a+9b & a+5b & b & 0 & 0 & 0 & 0 \\
 a+5b &  0 & 0 & 0 & 0 & 0 & 0 \\
 b & 0 & 0 & 0 & 0 & 0 & 0 \\
 0 & 0 & 0 & 0 & 0 & 0 & 0 \\
 0 & 0 & 0 & 0 &  0 & 0 & 0 \\
 0 & 0 & 0 & 0 & 0 & 0 & 0 \\
 0 & 0 & 0 & 0 & 0 & 0 & 0 \\
\end{array}
\right).$$ This is the Hankel matrix for the sequence 
$$3a+9b, a+5b,b,0,0,0,\ldots.$$

\section{The case of $a C_{n+4}+ b C_{n+5}$}
In this section, we investigate the Hankel transform of $a C_{n+4}+ b C_{n+5}$. We start with the case $a=b=1$. The generating function of $C-{n+4}+C_{n+5}$ is given by
$$f(x)=\frac{1}{x^5}\left((1+x)c(x)-(1+2x+3x^2+7x^3+19x^4)\right).$$
Proceeding as before, we find that
$$\left(\frac{1}{(1+x)^2}, \frac{x}{(1+x)^2}\right)\cdot \mathcal{H} \cdot \left(\frac{1}{(1+y)^2}, \frac{y}{(1+y)^2}\right)^T$$ has bivariate generating function given by
$$\frac{x^4 + x^3(y + 9) + x^2(y^2 + 9y + 33) + x(y^3 + 9y^2 + 33y + 62) + y^4 + 9y^3 + 33y^2 + 62·y + 56}{1 - xy}.$$ 
This expands to give the $9$-diagonal matrix that begins
\begin{scriptsize}
$$\left(
\begin{array}{ccccccccccc}
 56 & 62 & 33 & 9 & 1 & 0 & 0 & 0 & 0 & 0 & 0 \\
 62 & 89 & 71 & 34 & 9 & 1 & 0 & 0 & 0 & 0 & 0 \\
 33 & 71 & 90 & 71 & 34 & 9 & 1 & 0 & 0 & 0 & 0 \\
 9 & 34 & 71 & 90 & 71 & 34 & 9 & 1 & 0 & 0 & 0 \\
 1 & 9 & 34 & 71 & 90 & 71 & 34 & 9 & 1 & 0 & 0 \\
 0 & 1 & 9 & 34 & 71 & 90 & 71 & 34 & 9 & 1 & 0 \\
 0 & 0 & 1 & 9 & 34 & 71 & 90 & 71 & 34 & 9 & 1 \\
 0 & 0 & 0 & 1 & 9 & 34 & 71 & 90 & 71 & 34 & 9 \\
 0 & 0 & 0 & 0 & 1 & 9 & 34 & 71 & 90 & 71 & 34 \\
 0 & 0 & 0 & 0 & 0 & 1 & 9 & 34 & 71 & 90 & 71 \\
 0 & 0 & 0 & 0 & 0 & 0 & 1 & 9 & 34 & 71 & 90 \\
\end{array}
\right).$$
\end{scriptsize}

This indicates that the Hankel transform of $C_{n+4}+C_{n+5}$ is given by 
$$[x^{n+1}]\frac{1+ 35x+ 160x^2 -120x^3 -371x^4+ 371x^5+ 120x^6 -160x^7 -35x^8 -x^9}{(1-3x+x^2)^7}.$$
In the general case of $aC_{n+4}+ b C_{n+5}$, we obtain the $9$-diagonal matrix that begins
\begin{scriptsize}
$$\left(
\begin{array}{ccccccccc}
 14 a+42 b & 14 a+48 b & 6 a+27 b & a+8 b & b & 0 & 0 & 0 & 0 \\
 14 a+48 b & 20 a+69 b & 15 a+56 b & 6 a+28 b & a+8 b & b & 0 & 0 & 0 \\
 6 a+27 b & 15 a+56 b & 20 a+70 b & 15 a+56 b & 6 a+28 b & a+8 b & b & 0 & 0 \\
 a+8 b & 6 a+28 b & 15 a+56 b & 20 a+70 b & 15 a+56 b & 6 a+28 b & a+8 b & b & 0 \\
 b & a+8 b & 6 a+28 b & 15 a+56 b & 20 a+70 b & 15 a+56 b & 6 a+28 b & a+8 b & b \\
 0 & b & a+8 b & 6 a+28 b & 15 a+56 b & 20 a+70 b & 15 a+56 b & 6 a+28 b & a+8 b \\
 0 & 0 & b & a+8 b & 6 a+28 b & 15 a+56 b & 20 a+70 b & 15 a+56 b & 6 a+28 b \\
 0 & 0 & 0 & b & a+8 b & 6 a+28 b & 15 a+56 b & 20 a+70 b & 15 a+56 b \\
 0 & 0 & 0 & 0 & b & a+8 b & 6 a+28 b & 15 a+56 b & 20 a+70 b \\
\end{array}
\right).$$
\end{scriptsize}
The principal minor sequence of this matrix gives the Hankel transform of $a C_{n+4}+b C_{n+5}$. We deduce that the Hankel transform $h_n$ is given by
$$h_n=[x^{n+1}] \frac{A+Bx+Cx^2+Dx^3+Ex^4+Fx^5+Gx^6+Hx^7+Ix^8+Jx^9}{(1-(a+2b)x+b^2x^2)^7},$$
where
\begin{align*}
A&=1\\
B&=7a+28b\\
C&=7a^2+56ab+97b^2\\
D&=a^3+12a^2b+6ab^2-139b^3\\
E&=-7a^2b^2-91ab^3-273b^4\\
F&=-Eb\\
G&=-Db^3\\
H&=-Cb^5\\
I&=-Bb^7\\
J&=-Ab^9.\end{align*}
When $a=b=1$, we have that the Hankel transform of $C_{n+4}+C_{n+5}$ is given by
$$h_n=[x^{n+1}]\frac{1+35x+160x^2-120x^3-371x^4+371x^5+120x^6-160x^7-35x^8-x^9}{(1-3x+x^2)^7}.$$

We finish this section by noting that the matrix
$$ \tilde{M}\cdot H(n+4,a,b)\cdot \tilde{M}^T-M\cdot H(n+3,a,b)\cdot M^T $$ begins
$$\left(
\begin{array}{ccccccc}
 9a+28b & 5a+20b & a+7b & b & 0 & 0 & 0 \\
 5a+20b &  a+7b & b & 0 & 0 & 0 & 0 \\
 a+7b & b & 0 & 0 & 0 & 0 & 0 \\
 b & 0 & 0 & 0 & 0 & 0 & 0 \\
 0 & 0 & 0 & 0 &  0 & 0 & 0 \\
 0 & 0 & 0 & 0 & 0 & 0 & 0 \\
 0 & 0 & 0 & 0 & 0 & 0 & 0 \\
\end{array}
\right).$$ This is the Hankel matrix for the sequence
$$9a+28b , 5a+20b , a+7b , b,0,0,0,\ldots.$$ 

\section{The $2m+1$-diagonal matrices}
 An important idea of \cite{Dougherty} was to translate the calculation of Hankel determinants to that of calculating the principal minors of an $2m+1$-diagonal matrix. In examining the $2m+1$-diagonal matrices above, we see that each one, after a finite number of columns, coincides with the symmetric $2m+1$-diagonal matrix with ``spine'' given by the columns of the following matrix, whose general element is
$$D_{n,m}= a \binom{2m}{n-1}+b \binom{2m+2}{n}.$$ 
$$\left(
\begin{array}{ccccccc}
 b & b & b & b & b & b & b \\
 a+2 b & a+4 b & a+6 b & a+8 b & a+10 b & a+12 b & a+14 b \\
 b & 2 a+6 b & 4 a+15 b & 6 a+28 b & 8 a+45 b & 10 a+66 b & 12 a+91 b \\
 0 & a+4 b & 6 a+20 b & 15 a+56 b & 28 a+120 b & 45 a+220 b & 66 a+364 b \\
 0 & b & 4 a+15 b & 20 a+70 b & 56 a+210 b & 120 a+495 b & 220 a+1001 b \\
 0 & 0 & a+6 b & 15 a+56 b & 70 a+252 b & 210 a+792 b & 495 a+2002 b \\
 0 & 0 & b & 6 a+28 b & 56 a+210 b & 252 a+924 b & 792 a+3003 b \\
 0 & 0 & 0 & a+8 b & 28 a+120 b & 210 a+792 b & 924 a+3432 b \\
 0 & 0 & 0 & b & 8 a+45 b & 120 a+495 b & 792 a+3003 b \\
 0 & 0 & 0 & 0 & a+10 b & 45 a+220 b & 495 a+2002 b \\
 0 & 0 & 0 & 0 & b & 10 a+66 b & 220 a+1001 b \\
 0 & 0 & 0 & 0 & 0 & a+12 b & 66 a+364 b \\
 0 & 0 & 0 & 0 & 0 & b & 12 a+91 b \\
 0 & 0 & 0 & 0 & 0 & 0 & a+14 b \\
 0 & 0 & 0 & 0 & 0 & 0 & b \\
\end{array}
\right).$$ 
We must therefore determine to what extent the matrices $\tilde{M}\cdot H(m, a,b)\cdot \tilde{M}^T$ differ from such symmetric matrices. For instance, for $a C_{n+4}+ b C_{n+5}$, the difference
$$\text{Diag}(a \binom{8}{n-1}+b \binom{10}{n})-\tilde{M}\cdot H(4, a,b)\cdot \tilde{M}^T$$ is given by 
$$\left(
\begin{array}{ccccccc}
 28a+120b & 8a+45b & a+10b & b & 0 & 0 & 0 \\
 8a+45b &  a+10b & b & 0 & 0 & 0 & 0 \\
 a+10b & b & 0 & 0 & 0 & 0 & 0 \\
 b & 0 & 0 & 0 & 0 & 0 & 0 \\
 0 & 0 & 0 & 0 &  0 & 0 & 0 \\
 0 & 0 & 0 & 0 & 0 & 0 & 0 \\
 0 & 0 & 0 & 0 & 0 & 0 & 0 \\
\end{array}
\right).$$ 
This is the Hankel matrix for the sequence that begins 
$$28a+120b , 8a+45b , a+10b , b , 0 , 0 , 0,\ldots.$$ 
This is the sequence $$a\binom{2\cdot 4}{3-n-1}+b\binom{2\cdot 4+2}{3-n}.$$ 
In general, the Hankel matrices that we must subtract from the symmetric matrices are determined by the sequences occurring as rows in the follow array:
$$\left(
\begin{array}{ccccccc}
 0 & 0 & 0 & 0 & 0 & 0 & 0 \\
 b & 0 & 0 & 0 & 0 & 0 & 0 \\
 a+6 b & b & 0 & 0 & 0 & 0 & 0 \\
 6 a+28 b & a+8 b & b & 0 & 0 & 0 & 0 \\
 28 a+120 b & 8 a+45 b & a+10 b & b & 0 & 0 & 0 \\
 120 a+495 b & 45 a+220 b & 10 a+66 b & a+12 b & b & 0 & 0 \\
 495 a+2002 b & 220 a+1001 b & 66 a+364 b & 12 a+91 b & a+14 b & b & 0 \\
 2002 a+8008 b & 1001 a+4368 b & 364 a+1820 b & 91 a+560 b & 14 a+120 b & a+16 b & b \\
\end{array}
\right).$$ 
This array has general element $a \binom{2 r}{(r-1)-n-1}+b \binom{2r+2}{(r-1)-n}$. 
Thus for instance we have that 
$$\text{Diag}(a \binom{2\cdot 5}{n-1}+b \binom{2\cdot 5+2}{n})-\tilde{M}\cdot H(5, a,b)\cdot \tilde{M}^T$$ is given by 
$$\left(
\begin{array}{ccccccc}
 120a+495b & 45a+220b & 10a+66b & a+12b & b & 0 & \cdots \\
 45a+220b &  10a+66b & a+12b & b & 0 & 0 & \cdots \\
 10a+66b & a+12b & b & 0 & 0 & 0 & \cdots \\
 a+12b & b & 0 & 0 & 0 & 0 & \cdots \\
 b & 0 & 0 & 0 &  0 & 0 & \cdots \\
 0 & 0 & 0 & 0 & 0 & 0 & \cdots \\
 \vdots & \vdots & \vdots & \vdots & \vdots & \vdots & \ddots \\
\end{array}
\right).$$ 

\section{Production matrices}
If $L$ is a lower-triangular invertible matrix, then the matrix 
$$P_L=L^{-1} \bar{L}$$ is a Hessenberg type matrix, called the production matrix of $L$ \cite{Deutsch}. 
The production matrix of a Riordan matrix takes a particularly simple form. Given a Hessenberg matrix $P$ (with non-zero super-diagonal), it is possible to reconstruct the matrix $L$ such that $P=P_L$.
In order to prove that the Hankel transform of $a C_n+ b C_{n+1}$ is given by
$$h_n=[x^{n+1}] \frac{1-bx}{1-(a+2b)x+b^2x^2},$$
the authors in \cite{Dougherty} make use of the fact that
$$\left(\frac{1}{1+x}, \frac{x}{(1+x)^2}\right) \cdot \mathcal{H} \cdot \left(\frac{1}{1+x}, \frac{x}{(1+x)^2}\right)^T$$ is tri-diagonal. We can recast the results of \cite{Dougherty} in the following form, where we let 
$$\mathcal{H}=\left(a C_{i+j}+bC_{i+j+1}\right)_{0 \le i,j \le \infty}.$$
\begin{proposition} Let $M$ be the lower-triangular matrix with production matrix $P$ defined by
$$\left(\frac{1}{1+x}, \frac{x}{(1+x)^2}\right) \cdot \mathcal{H} \cdot \left(\frac{1}{1+x}, \frac{x}{(1+x)^2}\right)^T.$$ Then $(-1)^n$ times the first column of
$$(M_{n,k}/b^k)^{-1}$$ is given by $\tilde{h}_n$, where $h_n$ is the Hankel transform of $a C_n+ bC_{n+1}$.
\end{proposition}
\begin{proof} The matrix $P$ begins 
$$\left(
\begin{array}{ccccccc}
 a+b & b & 0 & 0 & 0 & 0 & 0 \\
 b & a+2 b & b & 0 & 0 & 0 & 0 \\
 0 & b & a+2 b & b & 0 & 0 & 0 \\
 0 & 0 & b & a+2 b & b & 0 & 0 \\
 0 & 0 & 0 & b & a+2 b & b & 0 \\
 0 & 0 & 0 & 0 & b & a+2 b & b \\
 0 & 0 & 0 & 0 & 0 & b & a+2 b \\
\end{array}
\right).$$ This is of the form required to generate a Riordan array. The theory of Riordan arrays now shows us that $M$ is given by
$$\left(\frac{b(1+x)}{b+(a+2b)x+bx^2}, \frac{x}{b+(a+2b)x+bx^2}\right)^{-1}.$$ 
Then $(M_{n,k}/b^k)^{-1}$ is given by the Riordan array 
$$\left(\frac{1+bx}{1+(a+2b)x+b^2x^2}, \frac{x}{1+(a+2b)x+b^2x^2}\right).$$ 
The first column of this matrix then has generating function $\frac{1+bx}{1+(a+2b)x+b^2x^2}$. 
\end{proof}
We note that the production matrix of $(M_{n,k}/b^k)^{-1}$ begins
$$\left(
\begin{array}{cccc}
 a+b & 1 & 0 & 0 \\
 b^2 & a+2 b & 1 & 0 \\
 0 & b^2 & a+2 b & 1 \\
 0 & 0 & b^2 & a+2 b \\
\end{array}
\right).$$
In like fashion, we have the following proposition.
\begin{proposition} Let $H=\left(a C_{i+j+1}+bC_{i+j+2}\right)_{0 \le i,j \le \infty}$. Let $M$ be the lower-triangular matrix with production matrix given by
$$\left(\frac{1}{(1+x)^2}, \frac{x}{(1+x)^2}\right) \cdot \mathcal{H} \cdot \left(\frac{1}{(1+x)^2}, \frac{x}{(1+x)^2}\right)^T.$$ Then $(-1)^n$ times the first column of
$$(M_{n,k}/b^k)^{-1}$$ is given by $\tilde{h}_n$, where $h_n$ is the Hankel transform of $a C_{n+1}+ bC_{n+2}$.
\end{proposition}
We take the example of $2 C_{n+1} + 3 C_{n+2}$. This sequence begins
$$8, 19, 52, 154, 480, 1551, 5148,\ldots.$$
We have
$$\left(
\begin{array}{ccccc}
 1 & 0 & 0 & 0 & 0 \\
 -2 & 1 & 0 & 0 & 0 \\
 3 & -4 & 1 & 0 & 0 \\
 -4 & 10 & -6 & 1 & 0 \\
 5 & -20 & 21 & -8 & 1 \\
\end{array}
\right)\cdot \left(
\begin{array}{ccccc}
 8 & 19 & 52 & 154 & 480 \\
 19 & 52 & 154 & 480 & 1551 \\
 52 & 154 & 480 & 1551 & 5148 \\
 154 & 480 & 1551 & 5148 & 17446 \\
 480 & 1551 & 5148 & 17446 & 60112 \\
\end{array}
\right)\cdot \left(
\begin{array}{ccccc}
 1 & 0 & 0 & 0 & 0 \\
 -2 & 1 & 0 & 0 & 0 \\
 3 & -4 & 1 & 0 & 0 \\
 -4 & 10 & -6 & 1 & 0 \\
 5 & -20 & 21 & -8 & 1 \\
\end{array}
\right)^{\text{T}}$$
$$=\left(
\begin{array}{ccccc}
 8 & 3 & 0 & 0 & 0 \\
 3 & 8 & 3 & 0 & 0 \\
 0 & 3 & 8 & 3 & 0 \\
 0 & 0 & 3 & 8 & 3 \\
 0 & 0 & 0 & 3 & 8 \\
\end{array}\right).$$
This last matrix (extended infinitely) is the production matrix of the lower-triangular matrix that begins
$$\left(
\begin{array}{ccccc}
 1 & 0 & 0 & 0 & 0 \\
 8 & 3 & 0 & 0 & 0 \\
 73 & 48 & 9 & 0 & 0 \\
 728 & 630 & 216 & 27 & 0 \\
 7714 & 7872 & 3699 & 864 & 81 \\
\end{array}
\right).$$ Dividing the $k$-th column by $3^k$ gives us the matrix
$$\left(
\begin{array}{ccccc}
 1 & 0 & 0 & 0 & 0 \\
 8 & 1 & 0 & 0 & 0 \\
 73 & 16 & 1 & 0 & 0 \\
 728 & 210 & 24 & 1 & 0 \\
 7714 & 2624 & 411 & 32 & 1 \\
\end{array}
\right).$$ The production matrix of this array is given by
$$\left(
\begin{array}{ccccc}
 8 & 1 & 0 & 0 & 0 \\
 9 & 8 & 1 & 0 & 0 \\
 0 & 9 & 8 & 1 & 0 \\
 0 & 0 & 9 & 8 & 1 \\
 0 & 0 & 0 & 9 & 8 \\
\end{array}\right)$$ signifying that the inverse of this matrix is the Riordan array
$$\left(\frac{1}{1+8x+9x^2}, \frac{x}{1+8x+9x^2}\right).$$
This begins
$$\left(
\begin{array}{ccccc}
 1 & 0 & 0 & 0 & 0 \\
 -8 & 1 & 0 & 0 & 0 \\
 55 & -16 & 1 & 0 & 0 \\
 -368 & 174 & -24 & 1 & 0 \\
 2449 & -1616 & 357 & -32 & 1 \\
\end{array}
\right).$$
The Hankel transform of $2 C_{n+1}+3 C_{n+2}$ begins $8,55,368,2449,\ldots$.

In the general case of $a C_{n+1}+ b C_{n+2}$, we obtain that $\tilde{\mathcal{M}} \cdot \mathcal{H} \cdot \tilde{\mathcal{M}}^T$ begins
$$\left(
\begin{array}{cccccc}
 a+2 b & b & 0 & 0 & 0 & 0 \\
 b & a+2 b & b & 0 & 0 & 0 \\
 0 & b & a+2 b & b & 0 & 0 \\
 0 & 0 & b & a+2 b & b & 0 \\
 0 & 0 & 0 & b & a+2 b & b \\
 0 & 0 & 0 & 0 & b & a+2 b \\
\end{array}
\right).$$
This is the production matrix of the array that begins
$$\left(
\begin{array}{cccc}
 1 & 0 & 0 & 0 \\
 a+2 b & b & 0 & 0 \\
 a^2+4 b a+5 b^2 & 4 b^2+2 a b & b^2 & 0 \\
 a^3+6 b a^2+15 b^2 a+14 b^3 & 14 b^3+12 a b^2+3 a^2 b & 6 b^3+3 a b^2 & b^3 \\
\end{array}
\right).$$
Dividing the $k$-th column by $b^k$, we then get the array which begins
$$\left(
\begin{array}{cccc}
 1 & 0 & 0 & 0 \\
 a+2 b & 1 & 0 & 0 \\
 a^2+4 b a+5 b^2 & 2 a+4 b & 1 & 0 \\
 a^3+6 b a^2+15 b^2 a+14 b^3 & 3 a^2+12 b a+14 b^2 & 3 a+6 b & 1 \\
\end{array}
\right).$$
This array has a production matrix that begins
$$\left(
\begin{array}{ccccc}
 a+2 b & 1 & 0 & 0 & 0 \\
 b^2 & a+2 b & 1 & 0 & 0 \\
 0 & b^2 & a+2 b & 1 & 0 \\
 0 & 0 & b^2 & a+2 b & 1 \\
 0 & 0 & 0 & b^2 & a+2 b \\
\end{array}
\right).$$

The inverse of the re-scaled matrix is thus given by the Riordan array
$$\left(\frac{1}{1+(a+2b)x+b^2x^2}, \frac{x}{1+(a+2b)x+b^2 x^2}\right).$$
Then $(-1)^n$ times the first column of this matrix has generating function $\frac{1}{1-(a+2b)x+b^2x^2}$.

\section{Continued fractions and ratios of Hankel transforms}
The theory of orthogonal polynomials and continued fractions underlies much of the theory surrounding the Hankel transforms of the sums of successive pairs of Catalan numbers. The following example gives some insight into this. 
\begin{example}
We consider the case of $C_{n+2}+C_{n+3}$. This sequence begins 
$$7, 19, 56, 174, 561, 1859,\ldots.$$ We ``normalize" this sequence by dividing by $7$, to get the sequence 
$$1, 19/7, 56/7, 174/7, 561/7, 1859/7,\ldots.$$ 
This sequence then has a generating function that can be expressed as the continued fraction that begins 
$$\cfrac{1}{1-\frac{19}{7}x-\cfrac{\frac{31}{49}x^2}{1-\frac{489}{217}x-\cfrac{\frac{805}{961}x^2}{1-\frac{1511}{713}x-\cfrac{\frac{2418}{2645}x^2}{1-\cdots}}}}.$$
This continued fraction corresponds to the tridiagonal matrix that begins
$$\left(
\begin{array}{cccccccc}
 \frac{19}{7} & 1 & 0 & 0 & 0 & 0 & 0 & 0 \\
 \frac{31}{49} & \frac{489}{217} & 1 & 0 & 0 & 0 & 0 & 0 \\
 0 & \frac{805}{961} & \frac{1511}{713} & 1 & 0 & 0 & 0 & 0 \\
 0 & 0 & \frac{2418}{2645} & \frac{618}{299} & 1 & 0 & 0 & 0 \\
 0 & 0 & 0 & \frac{4807}{5070} & \frac{5549}{2717} & 1 & 0 & 0 \\
 0 & 0 & 0 & 0 & \frac{253045}{262086} & \frac{1650954}{813637} & 1 & 0 \\
 0 & 0 & 0 & 0 & 0 & \frac{14783406}{15155449} & \frac{92763259}{45894577} & 1 \\
 0 & 0 & 0 & 0 & 0 & 0 & \ddots & \ddots \\
\end{array}
\right).$$
The sequence of principal minors of this matrix then begins 
$$\frac{19}{7}, \frac{170}{31}, \frac{1075}{115}, \frac{5580}{390},\frac{25529}{1254}, \frac{107036}{3893},\frac{421035}{11789},\ldots.$$
These are the successive ratios of the Hankel transform of $C_{n+3}+C_{n+4}$ which begins
$$19, 170, 1075, 5580, 25529, 107036, 421035, 1577575,\ldots$$
and of the Hankel transform of $C_{n+2}+C_{n+3}$ which begins 
$$7, 31, 115, 390, 1254, 3893, 11789, 35045,\ldots.$$ 
In the general case of $a C_{n+2}+bC_{n+3}$, the sequence of principal minors begins
$$\frac{5 a+14 b}{2 a+5 b},\frac{14 a^2+72 a b+84 b^2}{3 a^2+14 a b+14 b^2}, \frac{30 a^3+220 a^2 b+495 a b^2+330 b^3}{4 a^3+27 a^2 b+54 a b^2+30 b^3},\ldots,$$
where the numerator polynomials give the Hankel transform of $a C_{n+2}+b C_{n+3}$, and the denominator polynomials give the Hankel transform of $a C_{n+3}+ b C_{n+4}$. 
We conjecture that this pattern holds good for all $a C_{n+m}+b C_{n+m+1}$. 
\end{example}

\section{Conclusions} 
We have used the techniques of generating functions and Riordan arrays to throw some light on the nature of the Hankel transforms of linear combinations of successive pairs of Catalan numbers. The form of such Hankel transforms in the main is still conjectural, tied up in the expression 
$$T_{n,k,m}=\frac{C_m \binom{m+k-2}{m-2}\binom{n+k+2m-2}{2k+2m-3}\prod_{j=0}^{\lfloor \frac{m-1}{2} \rfloor-1}\binom{2n+2m-2j-1}{2m-4j-5}\prod_{j=0}^{m-3}(2m-j-2)}{\binom{2m-2}{2m-3}\prod_{j=0}^{\lfloor \frac{m-1}{2} \rfloor-1} \binom{2m-2j-1}{2m-4j-5} \prod_{j=0}^{m-3}(n+k+2m-j-2)}$$ for the coefficient matrix that may determine the general case. The links that these matrices may have to interesting combinatorial objects (such as lattice paths and plane partitions) awaits further investigation.

\bigskip
\hrule
\bigskip
\noindent 2020 {\it Mathematics Subject Classification}: Primary
11C20; Secondary 11B83, 15A15, 15B05, 15B36.
\noindent \emph{Keywords:}  Catalan numbers, Hankel determinant, Riordan array, production matrix, generating function, integer sequence.

\bigskip
\hrule
\bigskip
\noindent (Concerned with sequences
\seqnum{A000108},\seqnum{A000108}, \seqnum{A001519}, \seqnum{A001906}, \seqnum{A007318}, \seqnum{A039598}, \seqnum{A039599}, \seqnum{A078812}, \seqnum{A078920}, \seqnum{A085478}, \seqnum{A123352}, and \seqnum{A197649}.)

\end{document}